\documentclass{amsart}

\usepackage{amssymb}
\usepackage{amsthm}
\usepackage[all]{xy}

\newtheorem{lemma}{Lemma}
\newtheorem{theorem}[lemma]{Theorem}

\newtheorem{example}[lemma]{Example}

\title[Limits of completely decomposable modules]{A note on the direct limit of increasing sequences of completely decomposable modules over integral domains}
\author{J. E. Mac\'{\i}as-D\'{\i}az}
\address{Departamento de Matem\'{a}ticas y F\'{\i}sica, Universidad Aut\'{o}noma de Aguascalientes, Avenida Universidad 940, Ciudad Universitaria, Aguascalientes, Ags. 20131, Mexico}
\email{jemacias@correo.uaa.mx}

\date{\today}

\begin{document}

\begin{abstract}
In this note, we establish conditions under which the union of an increasing sequence of completely decomposable modules over domains are again completely decomposable. In our investigation, the condition of purity of modules is crucial. In fact, the main result reported in this work states that a module is completely decomposable when it is the union of a countable, ascending chain of completely decomposable, pure submodules, providing thus a generalization of Hill's criterion of freeness from abelian group theory.
\end{abstract}

\maketitle

\section{Introduction\label{Sec:Intro}}

In 1970, P. Hill established that an abelian group is free if it is the union of a countable, increasing sequence of free, pure subgroups \cite {Hill}; with this result, Hill generalized a theorem previously proved by L. Pontryagin in 1934, which was valid only for countable groups \cite {Pontryagin}. Evidently, many avenues of algebraic research opened after Hill's result, particularly in the theory of modules over domains, where abelian group theory finds a natural route of expansion. For instance, the literature reports on some generalizations of Hill's theorem to freeness and projectivity of modules over domains \cite {Macias-PJM}, to completely decomposable and separable modules \cite {Fuchs-Macias, Rangas}, to modules which are direct sums of ideals of Pr\"{u}fer domains \cite {Macias-AC}, and to Butler modules \cite {Rangas}. However, most of these results impose restrictive conditions on the links of the countable chains and/or on the type of rings employed.

The present work establishes a generalization of Hill's theorem to completely decomposable modules over commutative, integral domains, which extends many of the results available in the literature on the topic. In particular, the main result of this note extends the version of Hill's theorem for completely decomposable modules presented in \cite {Fuchs-Macias, Macias-AC}. Section \ref {S:Prelim} of this note quotes some notions from module theory, such as the conditions of complete decomposability and purity, and some fundamental class of families of modules. That stage is used to present some crucial lemmas in our investigation. Finally, we close our study with the statement and the corresponding proof of our main result.

\section{Preliminaries\label{S:Prelim}}

Once and for all, we state that all modules in this work will be over an arbitrary \emph {domain}, that is, a commutative ring with identity and without zero divisors; for the sake of non-triviality, we assume explicitly that $1 \neq 0$. Moreover, we will restrict our attention to torsion-free modules. A submodule $N$ of a module $M$ over a domain $R$ is called \emph {pure} if every system of equations
\begin{equation*}
  \sum _{j = 1} ^n r _{i , j} x _j = a _i, \quad (i = 1 , \ldots , m),
\end{equation*}
where each $r _{i , j} \in R$ and $a _i \in N$, has a solution in $N$ whenever it is solvable in $M$. It is easy to see that every pure submodule $N$ of $M$ is also \emph {relatively divisible}, that is, the identity $r M \cap N = r N$ is satisfied for every $r \in R$. In fact, the two conditions coincide for modules over Pr\"{u}fer domains \cite{Warfield} and, among integral domains, Pr\"{u}fer domains are the only ones with such property \cite{Cartan-Eilenberg}.

By a \emph {completely decomposable} module we mean a torsion-free module which is the direct sum of rank-one submodules. Evidently, every free module over a domain $R$ is completely decomposable, and so are modules which are isomorphic to direct sums of ideals of $R$; in this sense, complete decomposability generalizes the notion of freedom in the category of $R$-modules (for elementary properties on purity, relative divisibility and complete decomposability of modules employed in this work without an explicit reference, we refer to \cite {Fuchs-Salce2}).

\begin{lemma}
  \label{Lemma:1}
  A module $M$ is completely decomposable if there exists a continuous, well-ordered, ascending chain
  \begin{equation}
    \label{Eq:rho-Chain}
    0 = N _0 < N _1 < \ldots < N _\rho < \ldots \quad (\rho < \sigma)
  \end{equation}
  of submodules of $M$, such that
  \begin{enumerate}
    \item[(a)] for every $\rho < \sigma$, $N _{\rho + 1} / N _\rho$ is completely decomposable, 
    \item[(b)] for every $\rho < \sigma$, the exact sequence $0 \rightarrow N _\rho \rightarrow N _{\rho + 1} \rightarrow N _{\rho + 1} / N _\rho \rightarrow 0$ splits, and
    \item[(c)] $M = \bigcup _{\rho < \sigma} N _\rho$.
  \end{enumerate}
\end{lemma}

\begin{proof}
  Under the hypotheses, $M$ is isomorphic to the direct sum of the modules $N _{\rho + 1} / N _\rho$ for every $\rho < \sigma$, thus, it is completely decomposable.
\end{proof}

Let $\kappa$ be an infinite cardinal number, and let $M$ be a module. By a \emph {$G ^* (\kappa)$-family} of submodules of $M$ we understand a family $\mathcal {B}$ of submodules of $M$ satisfying the following properties:
\begin{enumerate}
  \item[(i)] $0 , M \in \mathcal {B}$,
  \item[(ii)] $\mathcal {B}$ is closed under unions of ascending chains, and
  \item[(iii)] for every $H \subseteq M$ of cardinality at most $\kappa$ and every $A _0 \in \mathcal {B}$, there exists $A \in \mathcal {B}$ containing both $A _0$ and $H$, such that $A / A _0$ has rank at most $\kappa$.
\end{enumerate}

One immediately notices that every module has a $G ^* (\aleph _0)$-family of submodules, namely, the set of all its submodules. However, it is not true that every module has a $G ^* (\aleph _0)$-family consisting of pure submodules. Nevertheless, it has been shown that modules of projective dimension less than or equal to $1$ over Pr\"{u}fer domains with a countable number of maximal ideals (included valuation domains, obviously) do possess $G ^* (\aleph _0)$-families of pure submodules \cite{Bazzoni-Fuchs, Macias-PJM}.

\begin{example}
  \label {Ex:1}
  Every completely decomposable module $M$ can be given an explicit $G ^* (\aleph _0)$-family consisting of pure submodules. To show this, let $M = \oplus _{i \in I} M _i$ be a fixed, direct decomposition of $M$ into rank-one, torsion-free modules. Then the collection of submodules of $M$ defined by
  \begin{equation*}
    \mathcal {B} = \{ \oplus _{j \in J} M _j : J \subseteq I \},
  \end{equation*}
  which consists of all the partial direct sums in the given decomposition of $M$, is one such family. Moreover, every member of $\mathcal {B}$ is evidently a completely decomposable module. Furthermore, given $A , B \in \mathcal {B}$ with $A \leq B$, it follows that $B / A$ is completely decomposable. \qed
\end{example}

For the remainder of the present section, we assume that $M$ is a torsion-free module over a domain, for which there exists a continuous, well-ordered, ascending chain
\begin{equation}
  \label{Eq:kappa-Chain}
  0 = M _0 < M _1 < \ldots < M _\nu < \ldots \quad (\nu < \kappa)
\end{equation}
of submodules of $M$, satisfying the following properties:
\begin{enumerate}
  \item[(A)] every $M _\nu$ is pure in $M$,
  \item[(B)] every $M _\nu$ has a $G ^* (\kappa)$-family $\mathcal {B} _\nu$ consisting of pure submodules, and
  \item[(C)] $M = \bigcup _{\nu < \kappa} M _\nu$.
\end{enumerate}

The following lemmas are the cornerstones of our investigation. For their proofs, see Section XVI.1 of \cite {Fuchs-Salce2}.

\begin{lemma}
  \label {Lemma:2}
  There exists a continuous, well-ordered, ascending chain 
  \begin{equation}
    \label{Eq:tau-Chain}
    0 = A _0 < A _1 < \ldots < A _\alpha < \ldots \quad (\alpha < \tau)
  \end{equation}
  of submodules of $M$, such that
  \begin{enumerate}
    \item[(a)] for every $\alpha < \tau$ and $\nu < \kappa$, $A _\alpha \cap M _\nu \in \mathcal {B} _\nu$,
    \item[(b)] for every $\alpha < \tau$, the factor module $A _{\alpha + 1} / A _\alpha$ has rank at most $\kappa$,
    \item[(c)] for every $\alpha < \tau$ and $\nu < \kappa$, $(A _\alpha \cap M _{\nu + 1}) + (A _{\alpha + 1} \cap M _\nu) \in \mathcal {B} _{\nu + 1}$, and
    \item[(d)] $M = \bigcup _{\alpha < \tau} A _\alpha$. \qed
  \end{enumerate}
\end{lemma}

Departing from the previous lemma, the proof of the following result is easy. One should only mention that the modules in the ascending chain \eqref {Eq:rho-Chain} described below, are actually the collection of modules $A _\alpha + (A _{\alpha + 1} \cap M _\nu)$, for every $\alpha < \tau$ and $\nu < \kappa$, well ordered under set inclusion.

\begin{lemma}
  \label{Lemma:3}
  There exists a continuous, well-ordered, ascending chain \eqref {Eq:rho-Chain} of submodules of $M$, such that
  \begin{enumerate}
    \item[(a)] every factor module $N _{\rho + 1} / N _\rho$ has rank at most $\kappa$,
    \item[(b)] every factor module $N _{\rho + 1} / N _\rho$ is isomorphic to a factor module $A / B$, with both $A$ and $B$ in some family $\mathcal {B} _\nu$ and $B < A$, and
    \item[(c)] $M = \bigcup _{\rho < \sigma} N _\rho$. \qed
  \end{enumerate}
\end{lemma}

\section{Main result}

The following is the most important result of this work. Clearly, it generalizes Hill's criterion of freeness of abelian groups, to a criterion for the complete decomposability of modules over domains.

\begin{theorem}
  \label{Thm:Main}
  A torsion-free module $M$ over a domain is completely decomposable if there exists a countable, ascending chain
  \begin{equation}
    \label {Eq:CountChain}
    0 = M _0 < M _1 < \ldots < M _n < \ldots \quad (n < \omega)
  \end{equation}
  of submodules of $M$, such that
  \begin{enumerate}
    \item[(a)] every $M _n$ is completely decomposable,
    \item[(b)] every $M _n$ is pure in $M$, and
    \item[(c)] $M = \bigcup _{n < \omega} M _n$.
  \end{enumerate}
\end{theorem}

\begin{proof}
  For each $n < \omega$, let $\mathcal {B} _n$ be the family consisting of all partial direct sums of modules in a fixed decomposition of $M _n$ as a direct sum of rank-one submodules. Example \ref {Ex:1} guarantees that such collection is a $G ^* (\aleph _0)$-family of pure submodules of $M _n$. Then, there exists a continuous, well-ordered, ascending chain \eqref {Eq:tau-Chain} satisfying the properties (a), (b), (c) and (d) of Lemma \ref {Lemma:2}. The modules $A _\alpha + (A _{\alpha + 1} \cap M _n)$ with $\alpha < \tau$ and $n < \omega$, well-ordered under set inclusion, form a continuous, ascending chain \eqref {Eq:rho-Chain} satisfying (a), (b) and (c) in Lemma \ref {Lemma:3}. In particular, \eqref {Eq:rho-Chain} satisfies (a) and (c) of Lemma \ref {Lemma:1}, and we only need to show that (b) also holds.

  Recall that $N _\rho$ and $N _{\rho + 1}$ take on the form $A _\alpha + (A _{\alpha + 1} \cap M _n)$ and $A _\alpha + (A _{\alpha + 1} \cap M _{n + 1})$, respectively, for some $\alpha < t$ and $n < \omega$. We have the following commutative diagram with exact rows and exact columns in which $\iota$ and $\iota ^\prime$ are inclusion homomorphisms, and the isomorphism in the bottom row is the result of applying the modular law and the first isomorphism theorem to the factor module of the right:
  \begin{equation*}
    \xymatrix{
      & 0 \ar[d] & 0 \ar[d] \\
      0 \ar[r] & (A _\alpha \cap M _{n + 1}) + (A _{\alpha + 1} \cap M _n) \ar[d]\ar[r]^\iota & A _\alpha + (A _{\alpha + 1} \cap M _n) \ar[d] \\
      0 \ar[r] & A _{\alpha + 1} \cap M _{n + 1} \ar[d]\ar[r]^{\iota ^\prime} & A _\alpha + (A _{\alpha + 1} \cap M _{n + 1}) \ar[d] \\
      0 \ar[r] & \displaystyle {\frac {A _{\alpha + 1} \cap M _{n + 1}} {(A _\alpha \cap M _{n + 1}) + (A _{\alpha + 1} \cap M _n)}} \ar[d]\ar[r] & \displaystyle {\frac {A _\alpha + (A _{\alpha + 1} \cap M _{n + 1})} {A _\alpha + (A _{\alpha + 1} \cap M _n)}} \ar[d]\ar[r] & 0 \\
      & 0 & 0}
  \end{equation*}
  Properties (a) and (c) of Lemma \ref {Lemma:2} state that $(A _\alpha \cap M _{n + 1}) + (A _{\alpha + 1} \cap M _n)$ and $A _{\alpha + 1} \cap M _{n + 1}$ both belong in $\mathcal {B} _{n + 1}$, so that the first row splits. It follows that the sequence $0 \rightarrow N _\rho \rightarrow N _{\rho + 1} \rightarrow N _{\rho + 1} / N _\rho \rightarrow 0$ splits and, thus, $M$ is completely decomposable by Lemma \ref {Lemma:1}.
\end{proof}


\begin{thebibliography}{10}

\bibitem{Bazzoni-Fuchs}
  S.~Bazzoni and L.~Fuchs.
  \newblock On modules of finite projective dimension over valuation domains.
  \newblock In {\em Proceedings of the Conference on Abelian Groups and Modules in Udine}, volume 287 of {\em CISM Courses and Lectures}, pages 361--371. Springer, 1984.

\bibitem{Cartan-Eilenberg}
  H.~Cartan and S.~Eilenberg.
  \newblock {\em {Homological Algebra}}.
  \newblock Princeton Landmarks in Mathematics. Princeton University Press, Princeton, New Jersey, 1st edition, 1999.

\bibitem{Fuchs-Macias}
  L.~Fuchs and J.~E. Mac\'{\i}as-D\'{\i}az.
  \newblock {On completely decomposable and separable modules over Pr\"{u}fer domains}.
  \newblock {\em J. Commut. Alg.}, 2(2):159--176, 2010.

\bibitem{Fuchs-Salce2}
  L.~Fuchs and L.~Salce.
  \newblock {\em Modules over non-Noetherian Domains}, volume~84 of {\em Mathematical Surveys and Monographs}.
  \newblock American Mathematical Society, Providence, Rhode Island, 1st edition, 2001.

\bibitem{Hill}
  P.~Hill.
  \newblock {On the freeness of abelian groups: a generalization of Pontryagin's theorem}.
  \newblock {\em Bull. Amer. Math. Soc.}, 76(5):1118--1120, 1970.

\bibitem{Macias-PJM}
  J.~E. Mac\'{\i}as-D\'{\i}az.
  \newblock {A generalization of the Pontryagin-Hill theorems to projective modules over Pr\"{u}fer domains}.
  \newblock {\em Pacif. J. Math.}, 246(2):391--405, 2010.

\bibitem{Macias-AC}
  J.~E. Mac\'{\i}as-D\'{\i}az.
  \newblock {On the unions of ascending chains of direct sums of ideals of $h$-local Pr\"{u}fer domains}.
  \newblock {\em Alg. Colloq.}, 2011.
  \newblock In press.

\bibitem{Pontryagin}
  L.~Pontryagin.
  \newblock The theory of topological commutative groups.
  \newblock {\em Annals of Math.}, 35(2):361--388, 1934.

\bibitem{Rangas}
  K.~M. Rangaswamy.
  \newblock A criterion for complete decomposability and {Butler} modules over valuation domains.
  \newblock {\em J. Algebra}, 205(1):105--118, 1998.

\bibitem{Warfield}
  R.~B. Warfield.
  \newblock Purity and algebraic compactness of modules.
  \newblock {\em Pacif. J. Math.}, 28:699--719, 1969.
\end{thebibliography}
\end{document}